\documentclass[12pt]{amsart}
\usepackage{amssymb, amsfonts, amsmath}
\usepackage[all]{xy}
\usepackage[shortlabels]{enumitem}
\usepackage{hyperref, color}
\hypersetup{
	colorlinks = true,
	linkcolor = red,
	filecolor = blue,
	urlcolor = blue,
	citecolor = blue
}
\usepackage{commath}
\usepackage{esdiff}
\usepackage{graphicx,subfigure}
\usepackage[center]{caption}
\usepackage[left = 2.3 cm, right = 2.3 cm, top = 2.4 cm, bottom = 2.4 cm]{geometry}

\newcommand{\R}{\mathbb R}

\DeclareMathOperator{\tr}{\mathrm{tr}}

\DeclareMathOperator{\intt}{int}

\newtheorem{theorem}{Theorem}[section]
\newtheorem{lemma}[theorem]{Lemma}

\newtheorem{corollary}[theorem]{Corollary}
\newtheorem*{conjecture}{Catenoid Conjecture}

\theoremstyle{definition}

\newtheorem{example}[theorem]{Example}
\newtheorem{remark}[theorem]{Remark}

\numberwithin{equation}{section}
\numberwithin{figure}{section}

\begin{document}
	
\title[A Note on Free Boundary Hypersurfaces in Space Forms Balls]{A Note on Free Boundary Hypersurfaces in Space Forms Balls}	
\author[Iury Domingos]{Iury Domingos}
\address{Universidade Federal da Paraíba\\
	Departamento de Matem\'atica\\
	58051-900. Jo\~ao Pessoa - PB, Brazil.}
\email{iury.domingos@im.ufal.br}

\author[Roney Santos]{Roney Santos}
\address{Universidade Federal de Minas Gerais\\
	Instituto de Ci\^encias Exatas\\		
	31270-901. Belo Horizonte - MG, Brazil.}
\email{roneysantos@mat.ufc.br}

\author[Feliciano Vit\'orio]{Feliciano Vit\'orio}	
\address{Universidade Federal de Alagoas\\
	Instituto de Matem\'{a}tica\\
	Campus A. C. Sim\~{o}es, BR 104 - Norte, Km 97, 57072-970.
	Macei\'o - AL, Brazil.}	
\email{feliciano@pos.mat.ufal.br}
	
\thanks{The second author was partially supported by CAPES/Brazil.}
	

	
\begin{abstract}
	In this article, we establish a relationship between geometric quantities of a hypersurface restricted to its boundary, and the geometric quantities of its boundary as a hypersurface of the boundary of the ball. As a first application, we prove that the quantity of umbilical points of a free boundary surface in the unit ball counted with multiplicities depend only on its topology; moreover, we obtain as consequences that free boundary surfaces are annuli if, and only if, they have no umbilical points, and a new proof of the Nitsche Theorem. Secondly, we prove two geometric integral inequalities for free boundary hypersurfaces, and use them to relate some geometric aspects of the hypersurface with topological aspects of its boundary in the three-dimensional case, and to give a new point of view to the Catenoid Conjecture.
		
\end{abstract}

\maketitle
\section{Introduction}
Minimal surface have been studied in Differential Geometry since 18th century. Initially motivated by variational problems, L. Euler and J. Lagrange considered the problem of finding a surface which is area-minimizing with a given boundary configuration. In 1849, J. Plateau, a mathematical physicist, verified that minimal surfaces can be obtained by immersing a wire frame into soapy water, and since then, the problem proposed by L. Euler and J. Lagrange became known as Plateau's Problem. In 1931, the existence of those surfaces when the boundary is a Jordan curve was showed independently by J. Douglas and T. Rad\'o, and it is a milestone in minimal surfaces theory.

In terms of the mean curvature, minimal surfaces have zero mean curvature. A generalisation of these surfaces are constant mean curvature surfaces. They also minimize area subject to a volume constraint.


The main goal of this work is to study free boundary constant mean curvature hypersurfaces in space forms balls. For our propose, a compact hypersurface $\Sigma,$ with non-empty boundary $\partial \Sigma$, into a ball $B$ in a space form $\mathbb{M}_c^{n + 1}$ is said to be a {\it free boundary} hypersurface in $B$ if $\intt(\Sigma) \subset \intt(B),$ $\partial \Sigma \subset \partial B$ and $\Sigma$ meets $\partial B$ orthogonally.

This kind of problem was firstly studied in the middle of 20th century by R. Courant and H. Lewy. They considered the area-minimizing problem among all surfaces with predetermined boundary lying in a surface. It is worthwhile to mention that it is a natural extension of Plateau's Problem.



In the unit ball $B$ centered at the origin of $\mathbb{R}^3,$ the simplest examples of free boundary constant mean curvature surfaces are the equatorial disk, some spherical caps and the critical catenoid (for this last surface, see \cite[Example 1.2]{fraser2012minimal}). Many examples of free boundary minimal surfaces were obtained for a large class of surfaces with different number of boundary connected components and genus (cf. \cite{folha2017free,fraser2016sharp,kapouleas2017free, MR4279102}). In the higher dimensional case, we also have examples of free boundary minimal hypersurfaces (cf. \cite{freidin2017free}).

There is interesting analogy between closed minimal submanifolds in the sphere and free boundary minimal submanifolds in the Euclidean ball. For example, the classical Nitsche Theorem states that a free boundary minimal disk in the unit ball centered at the origin of $\mathbb{R}^3$ is the equatorial disk (cf. \cite[Theorem 1]{nitsche1985stationary}), and the Almgren Theorem states that a minimal sphere in $\mathbb{S}^3$ is an equator of $\mathbb{S}^3$ (cf. \cite[Lemma 1]{almgren1966some}).


Due to the work of A. Fraser and R. Schoen, free boundary minimal surfaces have received much attention in the last decade. In \cite{fraser2011first} and \cite{fraser2016sharp}, they studied the spectrum of the Dirichlet-to-Neumann map on surfaces with boundary and showed that any smooth metric maximizing the first normalized Steklov eigenvalue must be realized by a free boundary minimal surface.




On the recent research in free boundary minimal surfaces, it is important to highlight the acclaimed conjecture about the uniqueness of critical catenoid due to A. Fraser and M. Li. In comparison with closed minimal surfaces in $\mathbb{S}^3$, this conjecture is analogous to the uniqueness of minimal torus in $\mathbb{S}^3$ (cf. \cite[Theorem 1]{brendle2013embedded}).

\begin{conjecture}[\cite{fraser2014compactnessof}]\label{catenoid}
    Up to congruences, the unique embedded free boundary minimal annulus in the Euclidean unit ball centered at the origin of $\mathbb{R}^3$ is the critical catenoid.
\end{conjecture}

This work is organized as follows. In section \ref{preliminares}, we quote some basic concepts and facts which will be used and addressed throughout this paper. The first fact is the Simons Inequality for constant mean curvature hypersurfaces in space forms. Next, we recall a version of a Hardy-type integral inequality, due to M. Batista, H. Mirandola and the third author, for free boundary constant mean curvature hypersurfaces in the unit ball in a space form with non-positive constant sectional curvature.

In section \ref{Stahl}, we establish two identities that relate the square norm of the umbilicity tensor of a free boundary hypersurface in a unit ball of a space form restricted to its boundary, and the umbilicity tensor of this boundary as a hypersurface of the sphere that delimits the ball. As a first consequence of this result, we prove that if the squared norm of the umbilicity tensor of hypersurface is a subharmonic function, then the boundary of the hypersurface is a totally umbilical hypersurface of the boundary of the ball and has constant mean curvature.


In section \ref{umbilical}, we describe a relation between the Euler characteristic and umbilical points of a free boundary surface in the unit ball of a space form. The proved relation has a analogous for closed surfaces with constant mean curvature in the sphere. As a consequence, we prove that the constant mean curvature free boundary annuli in the unit ball are the unique surfaces without umbilical points, and we obtain a new proof of the Nitsche Theorem \cite[Theorem 4.1]{souamstationary} for constant mean curvature surfaces in space forms balls.


Section \ref{gii} is dedicated to prove some geometric integral inequalities. Firstly, we integrate the Simons Inequality to find our first geometric inequality for free boundary constant mean curvature hypersurfaces in the unit ball in space forms, and we characterise the totally umbilical $n$-disks in the equality case. Applying this first inequality for the three-dimensional case, we can relate the geometry of this hypersurface and the topology of its connected boundary components. Secondly, we take inspiration in the work of G. Catino \cite{catino2016remark} that establish a geometric integral inequality for closed constant mean curvature hypersurfaces in space forms, and as consequence, he has reobtained the Hopf-Chern Theorem. In this direction, we first generalise \cite[Theorem 1.1]{catino2016remark} for constant mean curvature hypersurfaces with non-empty boundary. By means of this inequality, we use the free boundary condition to characterise topologically the equality in the two-dimensional case, and we provide a new perspective to the Catenoid Conjecture.

Throughout this article, we will omit the volume element in the integrals for sake of simplicity.

\subsection*{Acknowledgements}
We are grateful to L. Ambrozio and E. Barbosa for the relevant observations and discussions on the results.

\section{Preliminaries}\label{preliminares}

Let $\Sigma^n$ be an oriented smooth hypersurface in a space form $\mathbb{M}_c^{n + 1}$ with constant sectional curvature $c \in \mathbb{R}$. Let us denote by $A$ the Weingarten operator of $\Sigma.$ The mean curvature $H$ of $\Sigma$ is defined as the trace of $A.$ We also consider the umbilicity tensor $\phi$ of $\Sigma$, given by
\[\phi = A - \frac{H}{n} I,\]
where $I$ is the identity tensor on $\mathrm{T} \Sigma.$ Moreover, when $H$ is a constant function, we define the polynomial $p_H$ associated to $\Sigma$ by
\begin{equation*}
	p_H(t) = t^2 + \frac{n - 2}{\sqrt{n(n - 1)}}|H|t - \frac{H^2}{n} - nc,\ \ t \in \mathbb{R}.
\end{equation*}

The following result is the Simons Inequality for hypersurfaces with constant mean curvature. For sake of completeness, we present a proof for this result and characterise the equality when the hypersurface is not totally umbilical and has dimension greater than 2.

\begin{lemma}[Simons Inequality, \cite{alencar1994hypersurfaces, docarmo1994}]\label{lemasimon}
	Let $\Sigma^n$ be a hypersurface in $\mathbb{M}_c^{n + 1}.$ If $\Sigma$ has constant mean curvature $H$, then at the points of $\Sigma$ we have
	\begin{equation}\label{simon}
		|\phi|^2 p_H(|\phi|) \geq \frac{n+2}{n} |\nabla_\Sigma|\phi||^2 - \frac{1}{2} \Delta_\Sigma |\phi|^2.
	\end{equation}
	Moreover, if $n \geq 3$ and $\Sigma$ is not a totally umbilical hypersurface, the equality occurs if, and only if:
	\begin{itemize}
		\item[$\mathrm{(i)}$] $\Sigma$ is a catenoid when $H=0$ and $c\leq0$;
        \item[$\mathrm{(ii)}$] $\Sigma$ is either a Clifford torus or an Otsuki hypersurface when $H=0$ and $c>0$;
		\item[$\mathrm{(iii)}$] $\Sigma$ is a Delaunay hypersurface when $H\neq 0.$
	\end{itemize}
\end{lemma}

\begin{proof}
	Let $\{e_1, \ldots, e_n\}$ be an orthonormal frame of $\mathrm{T} \Sigma$ which diagonalizes $\phi$ at each point of $\Sigma$, i.e., $\phi e_i = \sigma_i e_i,$ $i = 1, \ldots, n$. Then (cf.  \cite[Equation 2.8]{cheng1977hypersurfaces})
	\begin{equation}\label{laplaphi}
		\frac{1}{2} \Delta_\Sigma |\phi|^2 = |\nabla_\Sigma \phi|^2 + \frac{1}{2} \sum_{i,j = 1}^n R_{ij} (\sigma_i - \sigma_j)^2,
	\end{equation}
	where $R_{ij}$ is the sectional curvature of $\Sigma$ along the plane generated by $\{e_i, e_j\}.$ If $\kappa_1, \ldots, \kappa_n$ denote the principal curvatures of $\Sigma,$ we have $\sigma_i = \kappa_i - H/n,$ and then the Gauss Equation implies that
	\begin{align*}
		R_{ij} &= \kappa_i \kappa_j+c = \left(\sigma_i + \frac{H}{n}\right) \left(\sigma_j + \frac{H}{n}\right)+c\\
		&= \sigma_i \sigma_j + \frac{H}{n} (\sigma_i + \sigma_j) + \frac{H^2}{n^2} + c.
	\end{align*}
	Moreover, since $\tr \phi = 0,$ by a direct computation, we get
	\[\sum_{i, j = 1}^n\sigma_i \sigma_j (\sigma_i - \sigma_j)^2 = \sum_{i, j = 1}^n\left(\sigma_i^3 \sigma_j - 2 \sigma_i^2 \sigma_j^2 + \sigma_i \sigma_j^3\right) = - 2|\phi|^4,\]	
	\[\sum_{i, j = 1}^n (\sigma_i + \sigma_j) (\sigma_i - \sigma_j)^2 = \sum_{i, j = 1}^n \left(\sigma_i^3 - 2\sigma_i^2 \sigma_j + \sigma_i \sigma_j^2 + \sigma_i^2 \sigma_j - 2 \sigma_i \sigma_j^2 + \sigma_j^3\right) = 2n \sum_{i = 1}^n \sigma_i^3,\]
	\[\sum_{i, j = 1}^n (\sigma_i - \sigma_j)^2 = \sum_{i, j = 1}^n \left(\sigma_i^2 - 2 \sigma_i \sigma_j + \sigma_j^2\right) = 2n |\phi|^2.\]
	Thus, by Kato Inequality, equation \eqref{laplaphi} becomes
	\begin{equation}\label{laplaphi2}
		\begin{aligned}
			\frac{1}{2} \Delta_\Sigma |\phi|^2 &= |\nabla_\Sigma \phi|^2 - |\phi|^4 + H \sum_{i = 1}^n \sigma_i^3 + \frac{H^2}{n} |\phi|^2+nc|\phi|^2\\
			&\geq \frac{n + 2}{n}|\nabla_\Sigma |\phi||^2 - |\phi|^4 + H \sum_{i = 1}^n \sigma_i^3 + \frac{H^2}{n} |\phi|^2+nc|\phi|^2,
		\end{aligned}
	\end{equation}
	We observe that, by Okumura Lemma,
	\[- \frac{n - 2}{\sqrt{n(n - 1)}} |\phi|^3\leq\sum_{i = 1}^n \sigma_i^3 \leq \frac{n - 2}{\sqrt{n(n - 1)}} |\phi|^3,\]
    and, consequently, inequality \eqref{laplaphi2} may be rewritten as
	\[\frac{1}{2} \Delta_\Sigma |\phi|^2 \geq \frac{n + 2}{n} |\nabla_\Sigma |\phi||^2 - |\phi|^2 \left(|\phi|^2 + \frac{n - 2}{\sqrt{n(n - 1)}} |H||\phi| - \frac{H^2}{n} - nc\right).\]
	
	We suppose that $n \geq 3$ and the equality in \eqref{simon} for $\phi$ not identically zero in $\Sigma$. The minimal case is characterised as a catenoid for $c \leq 0,$ and either as a Clifford torus or an Otsuki minimal hypersurface for $c > 0$ (cf. \cite[Theorem 3.1]{tam2009stability}, \cite[Theorem 1.1]{wang2018simons}). If $H \neq 0,$ since the equality also occurs in Okumura Inequality, we have that one of the principal curvatures has multiplicity $1$ and the others have multiplicity $n - 1,$ and consequently $\Sigma$ is a rotational hypersurface (cf. \cite[Theorem 4.2]{carmo1983rotation}).
\end{proof}

Another important result to our proposes is a Hardy-type integral inequality proved in \cite{batista2017hardy}. In general, this inequality holds for compact submanifolds (possibly with non-empty boundary) in Hadamard manifolds. Here, we will consider only the case where the ambient manifold is a space form with non-positive sectional curvature.

Let $\Sigma^n,$ $n \geq 3,$ be a compact free boundary hypersurface in a unit ball $B$ in the space form $\mathbb{M}_c^{n + 1},$ $c \leq 0.$ We denote by $r$ the distance function of $\mathbb{M}_c^{n +1}$ to the center of $B.$ Since $r \leq 1$ on $\Sigma$ and $r = 1$ on $\partial \Sigma,$ if we chose $p = 2$ and $\gamma = 0$ in \cite[Theorem 3.2]{batista2017hardy}, by a direct computation, we get the following:

\begin{lemma}[Batista--Mirandola--Vitório, \cite{batista2017hardy}]\label{bmv2}
	Let $\Sigma^n,$ $n \geq 3,$ be a free boundary hypersurface in a unit ball $B$ in the space form $\mathbb{M}_c^{n + 1},$ $c \leq 0$. If $\Sigma$ has constant mean curvature $H$ and $f \in C^1(\Sigma)$ is a non-negative function, we have
	\[\int_\Sigma f^2 \leq \frac{4}{n^2} \int_\Sigma |\nabla_\Sigma f|^2 + \frac{H^2}{n^2} \int_\Sigma f^2 + \frac{2}{n} \int_{\partial \Sigma} f^2.\]
	Moreover, the equality occurs if, and only if, $f$ vanishes on $\Sigma.$
\end{lemma}

Given a compact hypersurface $\Sigma$ into a ball $B$ in a Riemannian manifold $M$ such that $\partial \Sigma \neq \emptyset,$ $\intt(\Sigma) \subset \intt(B)$ and $\partial \Sigma \subset \partial B,$ we can restrict the geometric quantities of $\Sigma$ to $\partial \Sigma$ and see $\partial \Sigma$ as a hypersurface in $\partial B.$ Since our objective is the study of free boundary hypersurfaces in balls in space forms, it is important for our results to consider the space forms as conformally flat manifolds.


Let $B_R$ be the open Euclidean ball with radius $0 < R \leq \infty$ centered at the origin of $\mathbb{R}^{n + 1}$ and $\varphi: B_R \to \mathbb{R}$ be a radial function, meaning that there exists a smooth function $u: [0, R^2) \to \mathbb{R}$ such that $\varphi(x) = u(|x|^2).$ Notice that
\[\nabla \varphi = 2 u'(|x|^2)x,\]
where $\nabla \varphi$ denotes the gradient of $\varphi.$ We consider the metric $\widehat{g} = e^{2\varphi} g$, where $g$ is the Euclidean metric, and write $M = (B_R, \widehat{g}).$ Let $\nabla$ and $\widehat{\nabla}$ denote the Riemannian connections of $g$ and $\widehat{g},$ respectively. For all vector field $X\in\mathfrak{X}(B_R)$, we have that
\begin{align}\label{conforme}
	\nonumber \widehat{\nabla}_X x &= \nabla_X x + g(\nabla \varphi, X) x + g(\nabla \varphi, x) X - g(x, X) \nabla \varphi\\
	&= X + 2 u'(|x|^2) \big(g(x, X)x + |x|^2 X - g(x, X) x\big)\\
	\nonumber &= \mu X,
\end{align}
where $| \cdot |$ denotes the norm induced by $g$ and $\mu(x) = 1 + 2 u'(|x|^2)|x|^2$ is the {\it potential function} of $x.$

\begin{example}[Space Forms]\label{formasespaciais}
    Let $c\in\mathbb{R}$ and consider the function $\varphi_c: B_R \to \mathbb{R}$ defined by
    \[\varphi_c(x) = \log \frac{1}{1 + \frac{c}{4} |x|^2},\]
    where $R = \infty$ if $c \in \{0, 1\},$ and $R = 2$ if $c = - 1$. Then $\mathbb{R}^{n + 1} = (\mathbb{R}^{n + 1},  e^{2 \varphi_0}g)$ is the Euclidean space endowed with the standard Euclidean metric, $\mathbb{S}^{n + 1} \setminus \{p\} = (\mathbb{R}^{n + 1}, e^{2 \varphi_1} g)$ is the sphere with constant sectional curvature $1$ minus the pole $p$ and $\mathbb{H}^{n + 1} = (B_2, e^{2 \varphi_{- 1}} g)$ is the hyperbolic space with constant sectional curvature $- 1.$
\end{example}


\section{The Key Lemma}\label{Stahl}

From now on, we fix the space form $\mathbb{M}_c^{n + 1}= (B_R, \widehat{g})$ with $\widehat{g} = e^{2 \varphi_c} g$ and $c \in \{- 1, 0, 1\},$ as in the Example \ref{formasespaciais}. The following construction is inspired in the work of L. Al\'ias, J. de Lira and J. Miguel Malacarne (cf. \cite{alias2006}).

Let $r < R$ be a positive number chosen such that the closed ball $B = (\overline{B}_r, \widehat{g})$ has radius 1 in $\mathbb{M}_c^{n + 1},$ where $\overline{B}_r$ is the closed Euclidean ball centered at the origin of $\mathbb{R}^{n + 1}$ with radius $r.$ We consider a compact oriented smooth hypersurface $\Sigma$ into $B$ such that $\partial \Sigma$ is non-empty, $\intt(\Sigma) \subset \intt(B)$ and $\partial\Sigma \subset \partial B.$ If, respectively, $\alpha^{M}_N$ and $A^{M}_N$ denote the second fundamental form and the Weingarten operator of $N$ as a submanifold of the Riemannian manifold $M$, as particular cases we set $\alpha = \alpha^{\mathbb{M}_c^{n + 1}}_\Sigma,$ $A = A^{\mathbb{M}_c^{n + 1}}_\Sigma,$ $\widetilde{\alpha} = \alpha^{\partial B}_{\partial \Sigma}$ and $\widetilde{A} = A^{\partial B}_{\partial \Sigma}.$ Under those notations, for all vector fields $X,Y\in \mathfrak{X}(\partial\Sigma),$ we have that

\begin{align*}
	\widehat{\nabla}_X Y &= \nabla^{\partial B}_X Y + \alpha^{B}_{\partial B}(X,Y)\\
	&= \nabla^{\partial \Sigma}_X Y + \widetilde{\alpha}(X,Y) + \alpha^{B}_{\partial B}(X,Y)
\end{align*}
and
\begin{align*}
	\widehat{\nabla}_X Y &= \nabla^\Sigma_X Y + \alpha(X,Y)\\
	&= \nabla^{\partial \Sigma}_X Y + \alpha_{\partial\Sigma}^\Sigma(X,Y) + \alpha(X,Y),
\end{align*}
where $\nabla^{\partial B},$ $\nabla^{\partial \Sigma}$ and $\nabla^\Sigma$ are, respectively, the Riemannian connections of $\partial B,$ $\partial \Sigma$ and $\Sigma$ as submanifolds in $\mathbb{M}_c^{n + 1}$. Thus
\begin{equation}\label{alpha}
	\alpha^{\mathbb{M}_c^{n + 1}}_{\partial B}(X,Y) + \widetilde{\alpha}(X,Y) = \alpha_{\partial\Sigma}^\Sigma(X,Y) + \alpha(X,Y)
\end{equation}
along $\partial \Sigma.$
Since $\partial B$ is a sphere in $\mathbb{M}_c^{n + 1},$ there is $\kappa \in \mathbb{R}$ such that $A^{\mathbb{M}_c^{n + 1}}_{\partial B} = \kappa I,$ where $I$ is the identity tensor in $\mathrm{T}(\partial B).$ Therefore, if we denote $\widehat{g} = \langle\ ,\ \rangle,$ since $\langle x, x \rangle = 1,$ by \eqref{conforme},
\begin{equation}\label{curvaturaesfera}
	\kappa = \langle A^{\mathbb{M}_c^{n + 1}}_{\partial B} e_i, e_i \rangle = - \langle \left(\widehat{\nabla}_{e_i} x\right)^\top, e_i \rangle = - \mu_0,
\end{equation}	
where $\mu_0 = 1 + 2 u'(r^2)r^2$ and $(\, \cdot\, )^\top$ denotes the orthogonal projection of an arbitrary vector field of $\mathfrak{X} (\mathbb{M}_c^{n + 1})$ onto $\mathfrak{X}(\partial B).$ Note that $\mu_0>0$. Then, by \eqref{alpha}
\begin{equation}\label{AemdSigma}
	- \mu_0 \langle X, Y \rangle x + \langle \widetilde{A} X, Y \rangle \xi = \langle A_{\partial\Sigma}^\Sigma X, Y \rangle \nu + \langle A X,Y \rangle \eta,
\end{equation}
where $x$ is the position vector on $\overline{B}_r$, $\xi \in \mathfrak{X}(\partial B)$ is a unit conormal vector field along $\partial\Sigma$ as a hypersurface in $M,$ $\nu$ is the unit conormal vector field along $\partial\Sigma$ as a hypersurface in $\Sigma$ pointing outward and $\eta$ is a globally unit normal vector field defined on $\Sigma.$

\begin{lemma}\label{dSigma}
	Let $\Sigma^n$ be a free boundary hypersurface in the unit ball $B$ described above. If $\Sigma$ has constant mean curvature $H$, at the points of $\partial \Sigma$ we have
	\begin{itemize}
		\item[$\mathrm{(i)}$] $\displaystyle |\phi|^2 = |\widetilde{A}|^2 + (H - \widetilde{H})^2 - \frac{H^2}{n}$;
		\item[$\mathrm{(ii)}$] $\nu(|\phi|^2) = - 2 \mu_0 \left(|\widetilde{A}|^2 + (n + 1)(H - \widetilde{H})^2 - 2 H (H - \widetilde{H})\right)$,
	\end{itemize}
	where $\phi$ is the umbilicity tensor of $\Sigma,$ while $\widetilde{A}$ and $\widetilde{H}$ are, respectively, the Weingarten operator and the mean curvature of $\partial \Sigma$ as a hypersurface in $\partial B$.
\end{lemma}

\begin{proof}
	(i) Since $\Sigma$ is free boundary in $B$, we have that $x = \nu$ and $\xi = \eta$ on $\partial\Sigma.$ Thus, by equation \eqref{AemdSigma}, $A = \widetilde{A}$ along $\partial\Sigma$.
	Therefore, if $\{e_1, \ldots, e_{n-1}\}$ is an orthonormal frame on $\partial\Sigma$ given by eigenvectors of $\widetilde{A},$ and $\widetilde{\kappa}_1, \ldots, \widetilde{\kappa}_{n-1}$ are its respective eigenvalues, the matrix of $A$ in the orthonormal frame $\{e_1, \ldots, e_{n-1}, \nu\}$ is given by
	\[A=\left(
	\begin{array}{ccccc}
		\widetilde{\kappa}_1 & 0 & \cdots & 0 & \langle A\nu,e_1\rangle \\
		0 & \widetilde{\kappa}_2 & \cdots & 0 & \langle A\nu,e_2\rangle \\
		\vdots & \vdots & \ddots & \vdots & \vdots \\
		0 & 0 & \cdots & \widetilde{\kappa}_{n-1} & \langle A\nu,e_{n-1}\rangle \\
		\langle A\nu,e_1\rangle & \langle A\nu,e_2\rangle & \cdots & \langle A\nu,e_{n-1}\rangle & \langle A\nu,\nu\rangle \\
	\end{array}
	\right).\]
	Now, we consider the support function $f(x) = \langle x, \eta \rangle,$ $x \in \Sigma.$ Given $X \in \mathfrak{X}(\Sigma)$, by \eqref{conforme} we have $X(f) = \langle x, \widehat{\nabla}_X \eta \rangle.$ In particular, since $\Sigma$ is free boundary in $B$, for each $X \in \mathfrak{X}(\partial\Sigma),$
	\begin{align*}
		0 = \langle \nu, \widehat{\nabla}_X \eta \rangle = - \langle \nu, AX \rangle.
	\end{align*}
	Thus, along $\partial \Sigma,$ $\nu$ is a principal direction of $\Sigma$ and the orthonormal frame $\{e_1, \ldots, e_{n - 1}, \nu\}$ diagonalizes $A.$ Therefore
	\[A = \left(
	\begin{array}{cc}
	\widetilde{A} & 0\\
	0 & \kappa_n
	\end{array}
	\right),\]
	where $\kappa_n = \langle A \nu, \nu \rangle.$ Moreover, since $H = \widetilde{H} + \kappa_n,$ we obtain that $\kappa_n^2 = (H - \widetilde{H})^2,$ and
	\begin{align*}
	    |\phi|^2 &= |A|^2 - \frac{H^2}{n}\\
	    &= |\widetilde{A}|^2 + (H - \widetilde{H})^2 - \frac{H^2}{n}.
	\end{align*}
	
	\noindent (ii) Since $H$ is constant, we have
	\[\nu(|\phi|^2) = \nu(|A|^2) = 2\sum_{i = 1}^{n - 1} \kappa_i \nu(\kappa_i) - 2 \kappa_n \sum_{i = 1}^{n - 1}\nu(\kappa_i).\]
	Note that \cite[Theorem 2.4]{stahl1996convergence} holds for hypersurfaces in an arbitrary space forms. Thus, since $A$ is a Codazzi tensor on $\Sigma,$ by \eqref{curvaturaesfera} we have that $\nu(\kappa_i) = \mu_0(\kappa_n - \kappa_i)$ for $i = 1, \ldots, n - 1,$ and consequently
	\begin{align*}
		\nu(|\phi|^2) &= 2 \mu_0\left((H - \kappa_n)\kappa_n - |A|^2 + \kappa_n^2 - (n - 1) \kappa_n^2 + (H - \kappa_n)\kappa_n\right)\\
		&= - 2\mu_0\left(|\widetilde{A}|^2 + (n + 1)(H - \widetilde{H})^2 - 2H(H - \widetilde{H})\right).
	\end{align*}
\end{proof}

\begin{remark}\label{obsdSigma}
Using the relation $|\widetilde{A}|^2 = |\widetilde{\phi}|^2 + \widetilde{H}^2/(n - 1)$ in Lemma \ref{dSigma}, by a direct computation, we get
\begin{equation}\label{obsdSigma2}
    \begin{aligned}
    |\phi|^2 &= |\widetilde{\phi}|^2 + \frac{1}{n(n - 1)} \left((n - 1)H - n \widetilde{H}\right)^2\\
    \nu(|\phi|^2) &= - 2 \mu_0 \left(|\widetilde{\phi}|^2 + \frac{1}{n - 1} \left((n - 1) H - n \widetilde{H}\right)^2\right),
\end{aligned}
\end{equation}
where $\widetilde{\phi}$ is the umbilicity tensor of $\partial \Sigma$ as a hypersurface of $\partial B.$ In particular, $\nu(|\phi|^2) \leq 0$.
\end{remark}

Note that, as a consequence of Remark \ref{obsdSigma}, we have the following.

\begin{corollary}\label{subharmonica}
	Let $\Sigma^n$ be a free boundary hypersurface in the unit ball $B.$ If $\Sigma$ has constant mean curvature and $\Delta_\Sigma|\phi|^2 \geq 0,$ where $\phi$ is the umbilicity tensor of $\Sigma,$ then each connected component of $\partial \Sigma$ is a totally umbilical hypersurface of $\partial B$ with constant mean curvature.
\end{corollary}

\begin{proof}
Since $\Delta_\Sigma|\phi|^2 \geq 0$ and $\nu(|\phi|^2) \leq 0$, by the Divergence Theorem we have that
	\[0 \leq \int_\Sigma\Delta_\Sigma|\phi|^2 = \int_\Sigma\nu(|\phi|^2) \leq 0.\]
By \eqref{obsdSigma2}, we have that $\nu(|\phi|^2) = 0.$ Therefore $\widetilde{\phi} = 0$ and $n\widetilde{H} = (n-1)H$ along $\partial \Sigma.$
\end{proof}

\begin{remark}
	By Corollary \ref{subharmonica}, $\Sigma$ is a totally umbilical surface if $n=2,$ since $\phi = 0$ on $\partial \Sigma.$ Moreover, \cite[Proposition 1.6]{barbosa2018area} implies that $\Sigma$ is an equatorial disk if $H=0.$
\end{remark}

\section{Umbilical Points and Topology}\label{umbilical}

In this section, as an application of Lemma \ref{dSigma}, we will relate the geometry and the topology of free boundary surfaces with constant mean curvature in $B$.

Let $\Sigma^2$ be a Riemannian surface and $\varphi: \Sigma \to \R$ be a non-vanishing function such that $\varphi = |\psi| f,$ where $\psi$ is holomorphic and $f$ is smooth and positive. Given $p \in \Sigma$ a zero of $\varphi$, the ${\it multiplicity}$ of $\varphi(p)$ is the multiplicity of $\psi(p).$ The {\it order} of $\varphi$ in a subset $S\subset\Sigma$ is defined as the sum of the multiplicities of $\varphi$ in $S$. For this kind of function, we have the following lemma.

\begin{lemma}\label{laplaciano}
Given a compact Riemannian surface $\Sigma^2$, if $\nu(\varphi)\varphi^{- 1}$ is integrable along $\partial \Sigma$ then
	\[\int_\Sigma \Delta_\Sigma (\log \varphi) = - 2\pi \theta_\Sigma - \pi \theta_{\partial\Sigma} + \int_{\partial \Sigma} \nu(\varphi)\varphi^{- 1},\]
	where $\theta_\Sigma$ and $\theta_{\partial\Sigma}$ are the orders of $\varphi$ respectively in $\intt(\Sigma)$ and $\partial\Sigma.$
\end{lemma}

\begin{proof}
	Let $\{x_1, \ldots, x_\ell, x_{\ell+1}, \ldots, x_m\},$ $m \geq 0,$ be the set of points of $\Sigma$ such that $\varphi(x_i) = 0.$ We assume that $x_i\in\intt(\Sigma)$ for $i=1,\ldots,\ell$ and $x_i\in\partial\Sigma$ for $i=\ell+1,\ldots,m.$ Now, we consider a conformal coordinate $z_i$ around $x_i,$ with conformal factor $\lambda^2$, and we define
	\[\mathcal{B}_i(\delta) = \{x \in \Sigma\ :\ |z_i(x) - z_i(x_i)| < \delta\},\ \ i = 1, \ldots, m,\]
	such that $\mathcal{B}_i(\delta) \cap \mathcal{B}_j(\delta) = \emptyset$, for $i \neq j.$ We set
\[ \Sigma_\delta = \Sigma \setminus \left(\bigcup_{i = 1}^m \mathcal{B}_i(\delta)\right), \ \ \ M_\delta =\intt(\Sigma)\cap\left(\bigcup_{i=\ell+1}^m \partial\mathcal{B}_i(\delta)\right), \ \ \ N_\delta =\partial\Sigma\cap\left(\bigcup_{i=\ell+1}^m \partial\mathcal{B}_i(\delta)\right),\]
and we note that $\partial\Sigma_\delta = \left(\bigcup_{i = 1}^\ell \partial\mathcal{B}_i(\delta)\right) \cup M_\delta\cup \left(\partial \Sigma \setminus N_\delta\right)$.
Thus, by the Divergence Theorem,
	\begin{equation}\label{integral'}
		\int_{\Sigma_\delta} \Delta_\Sigma (\log \varphi)= \sum_{i = 1}^\ell \int_{\partial \mathcal{B}_i(\delta)} \nu_i(\log \varphi) + \sum_{i = \ell+1}^m\int_{M_\delta\cap\partial \mathcal{B}_i(\delta)}\nu_i(\log\varphi) + \int_{\partial \Sigma \setminus N_\delta} \nu(\varphi)\varphi^{- 1}.
	\end{equation}
	Let $r_i(x) = |z_i(x) - z_i(x_i)|$ and assume that $\varphi(x_i)$ has multiplicity $\theta_i.$ In this case, $\varphi$ takes the form $\varphi = r_i^{\theta_i} f_i$ on $\mathcal{B}_i(\delta)$ for some positive smooth function $f_i.$ Thus,
	\[\nu_i (\log \varphi) = - \frac{\theta_i}{\lambda} \frac{d}{dr_i}(\log r_i) + \nu_i (\log f_i) = - \frac{\theta_i}{\lambda r_i} + \nu_i (\log f_i).\]
	On one hand, for $x_i\in\intt(\Sigma),$ we get
	\begin{align*}
		\int_{\partial \mathcal{B}_i(\delta)} \nu_i(\log \varphi)  &= - \frac{\theta_i}{\delta} \int_{\partial \mathcal{B}_i(\delta)} \frac{1}{\lambda}  + \int_{\partial \mathcal{B}_i(\delta)} \nu_i(\log f_i) \\
		&= - \frac{\theta_i}{\delta} \int_{\{r_i = \delta\}} dz_i d \overline{z}_i + \int_{\partial \mathcal{B}_i(\delta)} \nu_i(\log f_i) \\
		&= - 2 \pi \theta_i + \int_{\partial \mathcal{B}_i(\delta)} \nu_i(\log f_i),
	\end{align*}
	and then, since $f_i > 0,$ we have that
	\begin{equation}\label{limiteinterior}
		\lim_{\delta \to 0} \int_{\partial \mathcal{B}_i(\delta)} \nu_i(\log \varphi) = -2\pi \theta_i.
	\end{equation}
	On the other hand, for $x_i\in\partial\Sigma,$ by the same argument used above, since the length of $\{r_i=\delta\}$ is approximate by $\pi\delta,$ we conclude that
	\begin{equation}\label{limitebordo}
		\lim_{\delta \to 0} \int_{M_\delta\cap\partial \mathcal{B}_i(\delta)} \nu_i(\log \varphi) = -\pi \theta_i.
	\end{equation}
	Consequently, replacing equations \eqref{limiteinterior} and \eqref{limitebordo} in equation \eqref{integral'}, since $\nu(\varphi)\varphi^{-1}$ is integrable,
	\[\int_\Sigma \Delta_\Sigma (\log \varphi) = - 2\pi \theta_\Sigma - \pi \theta_{\partial\Sigma} + \int_{\partial \Sigma} \nu(\varphi)\varphi^{- 1}.\]
\end{proof}

\begin{theorem}\label{umbilicaltheorem}
	Let $\Sigma^2$ be a compact free boundary surface in $B.$ If $\Sigma$ has constant mean curvature surface and is not a disk, then
	\[\chi(\Sigma) = -\frac{\theta_\Sigma}{2} - \frac{\theta_{\partial\Sigma}}{4},\]
	where $\theta_\Sigma$ and $\theta_{\partial\Sigma}$ are the orders of $|\phi|$ respectively in $\intt(\Sigma)$ and $\partial\Sigma.$
\end{theorem}

\begin{proof}
	On one hand, by \cite[Theorem 0]{eschenburg1988constant} and Gauss-Bonnet Theorem, since the geodesic curvature in the direction of $\nu$ of any connected component of $\partial \Sigma$ is equal $-\mu_0$ by \eqref{AemdSigma}, using the same notation in the proof of Lemma \ref{laplaciano},
	\begin{equation}\label{integral1}
		\begin{aligned}
			\int_\Sigma \Delta_\Sigma (\log |\phi|) &= \lim_{\delta \to 0} \int_{\Sigma_\delta} \Delta_\Sigma (\log |\phi|) = 2 \lim_{\delta \to 0} \int_{\Sigma_\delta} K\\
			&= 2 \int_\Sigma K = 4\pi \chi(\Sigma) - 2 \mu_0 |\partial \Sigma|.
		\end{aligned}
	\end{equation}
	On the other hand, by \eqref{obsdSigma2},
	\[\int_{\partial \Sigma} |\phi|^{- 2} \nu(|\phi|^2) d \sigma = - 4 \mu_0 |\partial \Sigma|.\]
	Consequently, since the set of umbilical points of $\Sigma$ on $\partial \Sigma$ is discrete, Lemma \ref{laplaciano} implies that
	\begin{equation}\label{integral2}
		\int_\Sigma \Delta_\Sigma (\log |\phi|) = - 2 \pi \theta_\Sigma - \pi \theta_{\partial\Sigma} - 2 \mu_0 |\partial \Sigma|.
	\end{equation}
	Combining \eqref{integral1} and \eqref{integral2}, we obtain $4 \chi(\Sigma) = - 2\theta_\Sigma - \theta_{\partial\Sigma}.$
\end{proof}

\begin{remark}
	A version of Theorem \ref{umbilicaltheorem} in the context of closed surfaces in space forms can be found at \cite[Equation (2)]{eschenburg1988constant}. Moreover, Theorem \ref{umbilicaltheorem} reobtain Nitsche Theorem for free boundary constant mean curvature surfaces in three-dimensional space forms balls.
\end{remark}

Once that the order of $|\phi|$ is a non-negative integer number, as an interesting application of Theorem \ref{umbilicaltheorem}, we have the following statement.

\begin{corollary}\label{corolarioumbilico}
	Let $\Sigma^2$ be a compact free boundary surface in $B.$ Admit that $\Sigma$ has constant mean curvature. Then $\Sigma$ is an annulus if, and only if, there are no umbilical points on $\Sigma.$
\end{corollary}

\begin{remark}
	We point out that Corollary \ref{corolarioumbilico} generalise \cite[Lemma 4.3]{li2019free} for space forms balls and proves its converse statement, both for surfaces with constant mean curvature.
\end{remark}

\section{Geometric Integral Inequalities}\label{gii}

\subsection{A First Inequality}\label{inspiradosimon}


We denote by $B$ the unit ball of a space form $\mathbb{M}^{n+1}_c,$  for $c \in \{- 1, 0\},$ such that $B$ is modeled by the Euclidean ball centered at the origin of $\mathbb{R}^{n + 1}$ with an appropriate radius, and by $C_{n, H}$ the constant
\[C_{n, H} = \frac{(n + 2)(n^2 - H^2)}{4n},\]
depending on the constant mean curvature $H$ and the dimension of the hypersurface $\Sigma^n$ contained in $B.$ We also recall the polynomial $p_H$ associated to $\Sigma$ defined by
\begin{equation*}
	p_H(t) = t^2 + \frac{n - 2}{\sqrt{n(n - 1)}}|H|t - \frac{H^2}{n} - nc,\ \ t \in \mathbb{R}.
\end{equation*}

\begin{theorem}\label{1desigualdade1}
	Let $\Sigma^n,$ $n \geq 3,$ be a free boundary hypersurface in $B$. If $\Sigma$ has constant mean curvature $H$, then
	\[\int_\Sigma |\phi|^2 \left(p_H(|\phi|) - C_{n, H}\right) \geq \frac{n}{2} \int_{\partial \Sigma} \left((H - \widetilde{H})^2 - |\widetilde{A}|^2 \right) + 2\mu_0 H \int_{\partial \Sigma} \widetilde{H} - \frac{3n - 2}{2n}H^2 |\partial \Sigma| + L(\mu_0),\]
    where
    \[L(\mu_0) = (\mu_0 - 1) \left((n + 1) \int_{\partial \Sigma} (H - \widetilde{H})^2 + \int_{\partial \Sigma} |\widetilde{A}|^2 - 2H^2 |\partial \Sigma| \right),\]
    $\phi$ is the umbilicity tensor of $\Sigma$, while $\widetilde{A}$ and $\widetilde{H}$ are, respectively, the Weingarten operator and the mean curvature of $\partial \Sigma$ as a hypersurface in $\partial B.$ Moreover, the equality occurs if, and only if, $\Sigma$ is a totally umbilical hypersurface.
\end{theorem}

\begin{proof}
Integrating Simons Inequality \eqref{simon}, by the Divergence Theorem and Lemma \ref{bmv2}
\begin{equation*}\label{intsimon}
	\begin{aligned}
		\int_\Sigma |\phi|^2 p_H(|\phi|) &\geq \frac{n+2}{n} \int_\Sigma |\nabla_\Sigma |\phi||^2 - \frac{1}{2} \int_{\partial \Sigma} \nu(|\phi|^2)\\
		&\geq \frac{n + 2}{n} \cdot \frac{n^2}{4} \left(\left(1 - \frac{H^2}{n^2}\right) \int_\Sigma |\phi|^2 - \frac{2}{n} \int_{\partial \Sigma} |\phi|^2\right) - \frac{1}{2} \int_{\partial \Sigma} \nu(|\phi|^2)\\
		&= C_{n, H} \int_\Sigma |\phi|^2 - \frac{n + 2}{2} \int_{\partial \Sigma} |\phi|^2 - \frac{1}{2} \int_{\partial \Sigma} \nu(|\phi|^2).
	\end{aligned}
\end{equation*}
Observe that Lemma \ref{dSigma} implies that along $\partial \Sigma,$
\begin{multline*}
    \frac{n + 2}{2} |\phi|^2 + \frac{1}{2}\nu(|\phi|^2) = \frac{n + 2 - 2 \mu_0}{2} |\widetilde{A}|^2 + \frac{n + 2 - 2(n + 1)\mu_0}{2}(H - \widetilde{H})^2\\
    - 2\mu_0 H \widetilde{H} - \frac{n + 2 - 4n \mu_0}{2n}H^2.
\end{multline*}
Therefore,
\begin{multline*}
	    \int_\Sigma |\phi|^2 \left(p_H(|\phi|) - C_{n, H}\right) \geq \frac{2(n + 1)\mu_0 - (n + 2)}{2} \int_{\partial \Sigma} (H - \widetilde{H})^2 + \frac{2 \mu_0 - (n + 2)}{2} \int_{\partial \Sigma} |\widetilde{A}|^2\\
	    + 2\mu_0 H \int_{\partial \Sigma} \widetilde{H} - \frac{4n \mu_0 - (n + 2)}{2n}H^2 |\partial \Sigma|.
    \end{multline*}
By a direct computation, we obtain the inequality. Note that, by Lemma \ref{bmv2}, equality occurs if, and only if, $\phi$ vanishes identically in $\Sigma.$
\end{proof}

\begin{corollary}
	Let $\Sigma^3$ be a free boundary hypersurface in $B.$ If $\Sigma$ has constant mean curvature $H$ and $\Gamma_1, \ldots, \Gamma_m$ denote the $m$ boundary connected components of $\Sigma,$ then
	\begin{equation*}
   	\int_\Sigma |\phi|^2 \left(p_H(|\phi|) - C_{3, H}\right) \geq 6 \pi \sum_{i = 1}^m \chi(\Gamma_i)+(2\mu_0-3) H \int_{\partial \Sigma} \widetilde{H}+\frac{H^2 - 9c - 9\mu_0^2}{3} |\partial \Sigma| + L(\mu_0),
	\end{equation*}
	where $\chi(M)$ and $\widetilde{H}$ are, respectively, the Euler characteristic of the given manifold $M$ and the mean curvature of $\partial \Sigma$ as a hypersurface in $\partial B.$ Moreover, the equality occurs if, and only if, $\Sigma$ is a totally umbilical hypersurface.
\end{corollary}

\begin{proof}
By the Gauss Equation and \eqref{AemdSigma}, for each $i = 1, \ldots, m,$
\[\widetilde{H}^2 = \widetilde{\kappa}_1^2 + \widetilde{\kappa}_2^2 + 2 \widetilde{\kappa}_1 \widetilde{\kappa}_2 = |\widetilde{A}|^2 + 2(\widetilde{K}_i - c - \mu_0^2),\]
where $\widetilde{\kappa}_1$ and $\widetilde{\kappa}_2$ are the principal curvatures of $\Gamma_i$ as a surface in $\partial B$, and $\widetilde{K}_i$ is the intrinsic curvature of $\Gamma_i.$ Thus, Theorem \ref{1desigualdade1} and the Gauss-Bonnet Theorem imply
\begin{align*}
   	\int_\Sigma |\phi|^2 \Big(p_H(|\phi|) - C_{3, H}\Big) - &L(\mu_0)\\
   	&\geq \frac{3}{2}\int_{\partial \Sigma} \left(H^2 - 2H \widetilde{H} + \widetilde{H}^2-|\widetilde{A}|^2\right) + 2\mu_0 H \int_{\partial \Sigma} \widetilde{H} - \frac{7}{6} H^2 |\partial \Sigma|\\
   	&= 3\sum_{i = 1}^m \int_{\Gamma_i}\widetilde{K}_i +(2\mu_0-3) H \int_{\partial \Sigma} \widetilde{H}+\frac{H^2 - 9c - 9\mu_0^2}{3} |\partial \Sigma|\\
	&= 6 \pi \sum_{i = 1}^m \chi(\Gamma_i) +(2\mu_0-3) H \int_{\partial \Sigma} \widetilde{H}+\frac{H^2 - 9c - 9\mu_0^2}{3} |\partial \Sigma|.
\end{align*}
\end{proof}

\subsection{A Second Inequality}\label{inspiradocatino}

From now on, we denote by $B$ the unit ball of a space form $\mathbb{M}^{n+1}_c,$  for $c \in \{- 1, 0, 1\},$ such that $B$ is modeled by the Euclidean ball centered at the origin of $\mathbb{R}^{n + 1}$ with an appropriate radius.

\begin{theorem}\label{desigualdade2n}
	Let $\Sigma^n,$ $n \geq 2,$ be a compact hypersurface in $\mathbb{M}_c^{n + 1},$ for $c \in \{- 1, 0, 1\},$ with boundary $\partial \Sigma.$ If $\Sigma$ is a non-totally umbilical hypersurface with constant mean curvature $H$ and $\nu(|\phi|^2)  |\phi|^{-\frac{n+2}{n}}$ is integrable along $\partial \Sigma,$ then
		\[-\int_\Sigma |\phi|^{\frac{n - 2}{n}} p_H(|\phi|) \leq \frac{1}{2}\int_{\partial \Sigma} |\phi|^{-\frac{n+2}{n}}\nu(|\phi|^2),\]
	where $\phi$ is the umbilicity tensor of $\Sigma.$ Moreover, if $n \geq 3,$ the equality occurs if, and only if:
	\begin{itemize}
		\item[$\mathrm{(i)}$] $\Sigma$ is a catenoid when $H=0$ and $c \in \{- 1, 0\}$;
		\item[$\mathrm{(ii)}$] $\Sigma$ is either a Clifford torus or an Otsuki hypersurface when $H=0$ and $c = 1$;
		\item[$\mathrm{(iii)}$] $\Sigma$ is a Delaunay hypersurface when $H\neq 0$.
	\end{itemize}
\end{theorem}

\begin{proof}
Since $\phi$ vanishes at most in a subset of null volume (cf. \cite[Lemma 2.2]{catino2016conformally}), given $\varepsilon > 0,$ we define the non-empty subset of $\Sigma$
\[\Omega_\varepsilon = \{x \in \Sigma \ : \ |\phi|(x)\geq \varepsilon\}\]
and the continuous function
\[f_{\varepsilon}(x) = \Bigg\{\begin{array}{cl}
	|\phi|(x) & \mbox{if\ } x \in \Omega_{\varepsilon}\\
	\varepsilon & \mbox{if\ } x \in \Sigma \setminus \Omega_{\varepsilon}.
\end{array}\]
Note that, by Green's Identity
\begin{equation}\label{green}
	\begin{aligned}
		\int_\Sigma \Delta_\Sigma |\phi|^2 f_\varepsilon^{-\frac{n+2}{n}} &= - \int_\Sigma \langle \nabla_\Sigma |\phi|^2, \nabla_\Sigma f_\varepsilon^{-\frac{n+2}{n}}\rangle + \int_{\partial \Sigma} \nu(|\phi|^2) f_\varepsilon^{-\frac{n+2}{n}}\\
		&= \frac{2(n+2)}{n} \int_\Sigma \langle \nabla_\Sigma |\phi|, \nabla_\Sigma f_\varepsilon \rangle |\phi| f_\varepsilon^{- \frac{2(n+1)}{n}} + \int_{\partial \Sigma} \nu(|\phi|^2)  f_\varepsilon^{-\frac{n+2}{n}}.
	\end{aligned}
\end{equation}
Thus, by \eqref{green} and Simons Inequality \eqref{simon}, we obtain
\begin{equation}\label{intlapacian}
	\begin{aligned}
		\frac{n+2}{n} \int_\Sigma \Big(|\nabla_\Sigma |\phi||^2 f_\varepsilon^{-\frac{n+2}{n}} - \langle \nabla_\Sigma |\phi|, &\nabla_\Sigma f_\varepsilon \rangle |\phi| f_\varepsilon^{- \frac{2(n+1)}{n}} \Big)\\
		&- \int_\Sigma |\phi|^2 p_H(|\phi|)  f_\varepsilon^{-\frac{n+2}{n}} \leq \frac{1}{2}\int_{\partial \Sigma} \nu(|\phi|^2)  f_\varepsilon^{-\frac{n+2}{n}}.
	\end{aligned}
\end{equation}
Since $f_\varepsilon = |\phi|$ on $\Omega_\varepsilon$ and $\nabla_\Sigma f_\varepsilon = 0$ on $\Sigma \setminus \Omega_\varepsilon,$
\begin{equation}\label{zero}
		\int_\Sigma \Big(|\nabla_\Sigma |\phi||^2 f_\varepsilon^{-\frac{n+2}{n}} - \langle \nabla_\Sigma |\phi|, \nabla_\Sigma f_\varepsilon \rangle |\phi| f_\varepsilon^{-\frac{2(n+1)}{n}} \Big) = \int_{\Sigma \setminus \Omega_\varepsilon} |\nabla_\Sigma |\phi||^2 \varepsilon^{-\frac{n+2}{n}} \geq 0,
\end{equation}
and, consequently, \eqref{intlapacian} becomes
\[- \int_\Sigma |\phi|^2 p_H(|\phi|)  f_\varepsilon^{-\frac{n+2}{n}} \leq \frac{1}{2}\int_{\partial \Sigma} \nu(|\phi|^2)  f_\varepsilon^{-\frac{n+2}{n}}.\]
Therefore, since $|\phi|^{\frac{n + 2}{n}} f_\varepsilon^{- \frac{n + 2}{n}} \to 1$ almost everywhere on $\Sigma$ when $\varepsilon \to 0,$ and $f_\varepsilon = |\phi|$ along $\partial \Sigma$ for all small enough $\varepsilon > 0,$ we have that
\[\int_\Sigma |\phi|^\frac{n-2}{n}\left(\frac{H^2}{n} - |\phi|^2 - \frac{n-2}{\sqrt{n(n-1)}} |H| |\phi| + nc\right) \leq \frac{1}{2}\int_{\partial \Sigma} |\phi|^{- \frac{n + 2}{n}}\nu(|\phi|^2).\]
Note that equality occurs if, and only if, also occurs in Simons Inequality and in \eqref{zero}. Thus, Lemma \ref{lemasimon} concludes our assertion.
\end{proof}

We give another proof of Theorem \ref{desigualdade2n} when $n = 2.$
	
\begin{theorem}\label{desigualdade22}
	Let $\Sigma^2$ be a compact constant mean curvature surface in $\mathbb{M}^3_c,$ for $c \in \{- 1, 0, 1\},$ with boundary $\partial \Sigma.$ If $\Sigma$ is a non-totally umbilical surface and $\nu(|\phi|^2) |\phi|^{- 2}$ is integrable along $\partial\Sigma,$ then
	\begin{equation}\label{ineq}
		- \int_\Sigma p_H(|\phi|) \leq \frac{1}{2} \int_{\partial \Sigma} |\phi|^{- 2} \nu(|\phi|^2).
	\end{equation}
	Moreover, the equality occurs if, and only if, $\Sigma$ has no umbilical point.
\end{theorem}

\begin{proof}
Let $\{x_1, \ldots, x_\ell, x_{\ell+1}, \ldots, x_m\},$ $m \geq 0,$ be the set of umbilical points of $\Sigma.$ We assume that $x_i\in\intt(\Sigma)$ for $i=1,\ldots,\ell$ and $x_i\in\partial\Sigma$ for $i=\ell+1,\ldots,m.$ Consider the neighborhoods $\mathcal{V}_1, \ldots, \mathcal{V}_m$ respectively of $x_1, \ldots, x_m$ such that $\mathcal{V}_i \cap \mathcal{V}_j = \emptyset$ if $i \neq j,$ and define $\overline{\Sigma} = \Sigma \setminus \cup_{i = 1}^m \mathcal{V}_i.$ Since $\phi \neq 0$ on $\overline{\Sigma},$ it follows by Lemma \ref{simon} and  the Divergence Theorem that
\[-2 \int_{\overline{\Sigma}} p_H(|\phi|) \leq \int_{\overline{\Sigma}} \Delta_\Sigma \log |\phi|^2 = \int_{\partial \overline{\Sigma}} |\phi|^{-2} \nu(|\phi|^2).\]
Thus, setting $M_i = \intt(\Sigma)\cap\partial\mathcal{V}_i$ for $i=\ell+1,\ldots,m$ and $N=\partial\Sigma\cap\left(\bigcup_{i=\ell+1}^\partial\partial\mathcal{V}_i\right)$, we have
\begin{align*}
	-2 \int_\Sigma p_H(|\phi|) + 2 \sum_{i = 1}^m &\int_{\mathcal{V}_i} p_H(|\phi|)\\
	&\leq \int_{\partial\Sigma\setminus N} |\phi|^{- 2}\nu(|\phi|^2)  - \sum_{i = 1}^\ell \int_{\partial \mathcal{V}_i} |\phi|^{- 2}\nu_i(|\phi|^2)  - \sum_{i = \ell + 1}^m \int_{M_i} |\phi|^{- 2}\nu_i(|\phi|^2) ,
\end{align*}
where $\nu$ is the outward conormal of $\partial \Sigma$ and $\nu_i$ is the outward conormal of $\partial \mathcal{V}_i.$ By regularity of the function $|\phi|^2,$ we can suppose that $\partial \mathcal{V}_i$ is a level set of $|\phi|^2.$ In this case, $\nu_i$ pointing to the direction of $\nabla_\Sigma |\phi|^2$ and, consequently, the inequality above becomes
\[- \int_\Sigma p_H(|\phi|) \leq \frac{1}{2} \int_{\partial \Sigma}|\phi|^{- 2} \nu(|\phi|^2) .\]
Moreover, the equality occurs if, and only if, there are no umbilical points in $\Sigma.$
\end{proof}

\begin{corollary}\label{igualdade}
	Let $\Sigma^2$ be a free boundary surface in $B.$ If $\Sigma$ has constant mean curvature $H,$ then $\Sigma$ is not a totally umbilical surface if, and only if, $\chi(\Sigma) \leq 0$. Moreover, equality occurs in \eqref{ineq} if, and only if, $\Sigma$ is an annulus.
\end{corollary}

\begin{proof}
Note that, by equation \eqref{AemdSigma}, the geodesic curvature of any boundary connected component in the direction of $\nu$ as a curve in $\Sigma$ is $- \mu_0$. Since $\Sigma$ has isolated umbilical points, by \eqref{obsdSigma2} we get
	\begin{equation}\label{integral}
		\int_{\partial \Sigma} |\phi|^{- 2} \nu(|\phi|^2) = - 4 \mu_0 |\partial \Sigma|.
	\end{equation} 
Moreover, by Gauss Equation, we can check that $|\phi|^2 = -2(K - c) + H^2/2,$ where $K$ is the Gauss curvature of $\Sigma,$ which means that $p_H(|\phi|)=-2K$. By this, equation \eqref{integral}, Theorem \ref{desigualdade22} and the Gauss-Bonnet Theorem we obtain

    \[2 \mu_0 |\partial \Sigma| \leq \int_\Sigma p_H(|\phi|) = 2 \left(\mu_0 |\partial \Sigma| - 2\pi \chi(\Sigma)\right).\]
    Therefore, $\chi(\Sigma) \leq 0.$ In particular, equality occurs in \eqref{ineq} if, and only if, $\chi(\Sigma) = 0.$
\end{proof}

\begin{remark}
    We point out that by \cite[Lemma 4.3]{li2019free}, in the Euclidean case, a minimal annulus in $B$ cannot have umbilical point. In this sense, the assertions of Corollary \ref{igualdade} reformulate the Catenoid Conjecture as a unique solution problem of equality case in \eqref{ineq}.
\end{remark}

\begin{remark}
	Note that Corollary \ref{igualdade} implies the Nitsche Theorem, and it proves Corollary \ref{corolarioumbilico}.\break
\end{remark}

\bibliographystyle{amsplain}
\bibliography{references}
\end{document}